\pdfoutput=1
\documentclass{amsart}
\makeatletter
	
	\renewcommand{\upchars@}{%
		\def\ss{SS}
		\def\i{I}
		\def\j{J}
		\def\ae{\AE}
		\def\oe{\OE}
		\def\o{\O}
		\def\aa{\AA}
		\def\l{\L}
		\def\Mc@{M{\scshape c}}}
	\providecommand{\Mc@}{Mc}
	\usepackage{geometry}
	\usepackage{microtype}
	\usepackage{inputenc}
	\inputencoding{utf8}
	\usepackage{comment}
\makeatother
\makeatletter
	\usepackage{amsmath}
	\usepackage{mathtools}
	\usepackage{amssymb}
	\usepackage{mathrsfs}
	\renewcommand{\implies}{\mspace{5mu plus 0mu minus 5mu}\Longrightarrow\mspace{5mu plus 0mu minus 5mu}}

	\newcommand{\flsubstack}[1]{\subarray{l}#1\endsubarray}

\makeatother
\usepackage{amsthm}
\newtheorem{thm}{Theorem}[section]
\newtheorem*{thm*}{Theorem}
\newtheorem{lem}[thm]{Lemma}
\newtheorem*{lem*}{Lemma}
\newtheorem{prop}[thm]{Proposition}
\newtheorem*{prop*}{Proposition}
\usepackage{enumerate}
\usepackage{array}
\usepackage{bbm}
\usepackage{xcolor}
\usepackage{url}
\usepackage{hyperref}
\hypersetup{%
	hidelinks,
	bookmarksnumbered=true,
	bookmarksopen=true,
	pdfauthor={Alberto Takase}}
\usepackage{mleftright}

\let\left\mleft
\let\right\mright
\let\Owidehat\widehat
\renewcommand{\widehat}[1]{\vphantom{#1}\smash{\Owidehat{#1}}}
\let\Owidetilde\widetilde
\renewcommand{\widetilde}[1]{\vphantom{#1}\smash{\Owidetilde{#1}}}
\begin{document}
\title[On the spectra of separable 2D almost Mathieu operators]{On the spectra of separable 2D almost Mathieu operators}
\author[A.~Takase]{Alberto Takase}
\address{Department of Mathematics, University of California, Irvine, CA~92697, USA}
\email{atakase@uci.edu}
\begin{comment}
\author[A.\ Gorodetski]{Anton Gorodetski}
\address{Department of Mathematics, University of California, Irvine, CA~92697, USA}
\email{asgor@math.uci.edu}
\thanks{}
\end{comment}
\begin{abstract}
We consider separable 2D discrete Schrödinger operators generated by 1D almost Mathieu operators.
For fixed Diophantine frequencies we prove that for sufficiently small couplings the spectrum must be an interval.
This complements a result by J.~Bourgain establishing that for fixed couplings the spectrum has gaps for some (positive measure) Diophantine frequencies.
Our result generalizes to separable multidimensional discrete Schrödinger operators generated by 1D quasiperiodic operators whose potential is analytic and whose frequency is Diophantine.
The proof is based on the study of the thickness of the spectrum of the almost Mathieu operator, and utilizes the Newhouse Gap Lemma on sums of Cantor sets.
\end{abstract}
\begin{comment}
We consider separable 2D discrete Schr\"odinger operators generated by 1D almost Mathieu operators. For fixed Diophantine frequencies we prove that for sufficiently small couplings the spectrum must be an interval. This complements a result by J. Bourgain establishing that for fixed couplings the spectrum has gaps for some (positive measure) Diophantine frequencies. Our result generalizes to separable multidimensional discrete Schr\"odinger operators generated by 1D quasiperiodic operators whose potential is analytic and whose frequency is Diophantine. The proof is based on the study of the thickness of the spectrum of the almost Mathieu operator, and utilizes the Newhouse Gap Lemma on sums of Cantor sets.
\end{comment}
\begin{comment}
Mathematics Subject Classification
47B36
47B39
28A80
\end{comment}
\date{\today}
\thanks{The project was supported in part by the NSF grant DMS--1855541 (PI - A.\ Gorodetski).}
\maketitle
\section{Introduction}\label{sec: Introduction}
\noindent%
The \textit{almost Mathieu operator} is the discrete Schrödinger operator $ H_{\lambda,\alpha,\omega} $ on $ \ell^2(\mathbb{Z}) $ defined by
\[
\textstyle[H_{\lambda,\alpha,\omega}\psi](n)=\psi(n+1)+\psi(n-1)+2\lambda\cos(2\pi(n\alpha+\omega))\psi(n)
\]
for every $ \psi\in\ell^2(\mathbb{Z}),n\in\mathbb{Z} $.
Here $ \lambda\in\mathbb{R} $ is the \textit{coupling}, $ \alpha\in\mathbb{R}/\mathbb{Z}\eqqcolon \mathbb{T} $ is the \textit{frequency}, and $ \omega\in\mathbb{T} $ is the \textit{phase}.
The operator has its origins in solid-state physics and the study of electrons.
The operator also has a connection to the quantum Hall effect---a Nobel-prize-worthy discovery by K.~von~Klitzing, G.~Dorda, and M.~Pepper in 1980 \cite{1980QuantizedHallEffect}.
The connection was made by D.~J.~Thouless, M.~Kohmoto, M.~P.~Nightingale, and M.~den~Nijs in 1982 \cite{1982QuantizedHallEffect}.
For a discussion of discrete Schrödinger operators and the almost Mathieu operator see the 2019 preprint by S.~Jitomirskaya \cite{2019AMO}.
For more discussions see \cite{2019QPSOp,2017DamanikReviewPaper,2017MarxJitomirskayaReviewPaper,2007JitomirskayaReviewPaper,2005LastReviewPaper,2005Bourgain,1982SimonReviewPaper,1982BellissardReviewPaper}.
\par
P.~G.~Harper in 1955 \cite{1955HarperPhysicsPaper1of2,1955HarperPhysicsPaper2of2} under the tutelage of R.~E.~Peierls established that electrons in a cubic lattice and under a magnetic field have nondiscrete nonevenly-spaced broadened energy values.
Harper used a tight-binding approximation now called the Harper model.
The Schrödinger operator governing the spectrum of the Harper model is the almost Mathieu operator $ H_{\lambda,\alpha,\omega} $ with $ \lambda=1 $.
Physicists sought to better understand the topological structure of the spectrum $ \Sigma_{\lambda,\alpha,\omega} $ of $ H_{\lambda,\alpha,\omega} $.
M.~Azbel in 1964 \cite{1964AzbelPhysicsPaper} conjectured and D.~Hofstadter in 1976 \cite{1976HofstadterPhysicsPaper} computationally supported that $ \Sigma_{\lambda,\alpha,\omega} $ has the characteristic of being either band-like for rational $ \alpha $ or fractal-like for irrational $ \alpha $.
Decades later A.~Avila and S.~Jitomirskaya in 2009 \cite{2009TheTenMartiniProblem} made the final contribution to the resolution of the Ten Martini Problem which sought to confirm the conjectured topological structure of $ \Sigma_{\lambda,\alpha,\omega} $.
If $ \alpha=\frac{p}{q}\in\mathbb{T}\cap\mathbb{Q} $, then $ \Sigma_{\lambda,\alpha,\omega} $ is a disjoint union of at most $ q $-many compact intervals.
If $ \alpha\in \mathbb{T}\setminus\mathbb{Q} $, then $ \Sigma_{\lambda,\alpha,\omega} $ is a Cantor set.
Also, $ \Sigma_{\lambda,\alpha,\omega} $ is independent of the phase $ \omega $ when the frequency $ \alpha $ is irrational.
This last fact can be found in the 1982 review paper by B.~Simon \cite{1982SimonReviewPaper} (where the name of the almost Mathieu operator was introduced) and applies to operators with quasiperiodic potentials; the definition of quasiperiodic potentials can be found within \textsc{subsection~\ref{subsec: Quasiperiodic Potentials and The Almost Mathieu Operator}}.
In general, phase independence of the spectrum for operators with quasiperiodic potentials follows from the fact that the underlying topological dynamical system $ (\mathbb{T}^b,\mathbb{Z}) $ is minimal \cite{2017DamanikReviewPaper}.
\par
Let $ d\ge 2 $ be a positive integer.
We consider the $ d $-dimensional analog operator $ \vphantom{H}\smash{\Owidehat{H}} $ generated by almost Mathieu operators.
Specifically, $ \vphantom{H}\smash{\Owidehat{H}} $ is the discrete Schrödinger operator on $ \ell^2(\mathbb{Z}^d) $ defined by
\[
\textstyle[\vphantom{H}\smash{\Owidehat{H}}\psi](n)=\big(\sum_{m\in \left\{e_1,\ldots,e_d\right\}}\psi(n+m)+\psi(n-m)\big)+\big(\sum_{k\in\left\{1,\ldots,d\right\}}2\lambda_k\cos(2\pi(n_k\alpha_k+\omega_k))\big)\psi(n)
\]
for every $ \psi\in\ell^2(\mathbb{Z}^d),n\in\mathbb{Z}^d $; $ \left\{e_1,\ldots,e_d\right\} $ is the standard basis.
The operator $ \vphantom{H}\smash{\Owidehat{H}} $ and the spectrum $ \vphantom{\Sigma}\smash{\Owidehat{\Sigma}}=\Sigma_{\lambda_1,\alpha_1,\omega_1}+\cdots+\Sigma_{\lambda_d,\alpha_d,\omega_d} $ are the main objects of study in this paper.
The theory of the almost Mathieu operator provides insight into the theory of $ \vphantom{H}\smash{\Owidehat{H}} $.
Consider the following theorem established by A.~Avila and D.~Damanik in 2008 \cite{2008AbsoluteContinuity}: If $ \lambda\ne 1 $ and $ \alpha\in\mathbb{T}\setminus\mathbb{Q} $, then the integrated density of states of the almost Mathieu operator is absolutely continuous.
Along with a theorem by Steinhaus, which states that the sum of two sets with positive Lebesgue-measure contains an open interval, we immediately obtain the following proposition.
\begin{prop}\label{prop: PROP}
Assume at least two among $ \lambda_1,\ldots,\lambda_d $ are not equal to $ 1 $ and $ \alpha_1,\ldots,\alpha_d\in\mathbb{T}\setminus\mathbb{Q} $.
Then $ \vphantom{\Sigma}\smash{\Owidehat{\Sigma}} $ has a dense interior.
\end{prop}
\begin{comment}
Indeed, for each $ x\in\widehat{\Sigma} $, there exists an open interval $ I $ such that $ I\subseteq B_1+\cdots+B_d\subseteq \widehat{\Sigma} $, where $ x=x_1+\cdots+x_d $ and $ x_k\in \Sigma_{\lambda_k,\alpha_k,\om_k}\eqdef \Sigma_k $ and $ B_k $ is a $ \Sigma_k $-ball centered at $ x_k $ with a sufficiently small radius.
\begin{proof}
Observe $ \widehat{\Sigma}=\Sigma_{\lambda_1,\alpha_1,\om_1}+\cdots+\Sigma_{\lambda_d,\alpha_d,\om_d} $.
Because $ x\in\widehat{\Sigma} $, there exists $ x_k\in \Sigma_{\lambda_k,\alpha_k,\om_k}\eqdef \Sigma_k $ for $ k=1,\ldots,d $ such that $ x=x_1+\cdots+x_d $.
For each $ k $, define $ B_k\defeq B(x_k,\frac{\ep}{d})\cap\Sigma_{k} $.
For each $ k $, let $ \nu_k $ be the density of states measure of $ H_{\lambda_k,\alpha_k,\om_k} $.
Observe $ \nu_k(B_k)>0 $.
By Theorem~\ref{thm: Introduction(2)}, $ B_p $ and $ B_q $ have positive Lebesgue-measure, where $ \lambda_p\ne 1 $ and $ \lambda_q\ne 1 $ and $ p\ne q $.
A theorem by Steinhaus states that the sum of two sets with positive Lebesgue-measure contains an open interval.
As a result, $ B_1+\cdots+B_d $ contains an open interval $ I $.
Fix $ y\in I $.
Then $ y\in \widehat{\Sigma}{}^{\circ} $ and $ |x-y|<\ep $.
\end{proof}
\end{comment}
\noindent%
More can be obtained in the small coupling and Diophantine frequency regime.
This is the context of the main theorem in this paper.
\begin{thm}\label{thm: MT}
Assume $ \alpha_1,\ldots,\alpha_d $ are irrational and satisfy a Diophantine condition.
There exists $ \varepsilon=\varepsilon(\alpha_1,\ldots,\alpha_d)>0 $ such that if $ 0<|\lambda_1|,\ldots,|\lambda_d|<\varepsilon $, then $ \vphantom{\Sigma}\smash{\Owidehat{\Sigma}} $ is an interval.
\end{thm}
\noindent%
We mention for comparison a theorem by J.~Bourgain in 2002 \cite{2002SquareAMO}.
\begin{thm}[J.~Bourgain \cite{2002SquareAMO}]\label{thm: B}
Assume $ \lambda_1,\ldots,\lambda_d=\lambda>0 $.
There exist positive Haar-measure sets $ \ooalign{$ A_{\lambda}$\cr
	$\phantom{A}^{(1)}$\cr
	$\phantom{A^{(1)}_{\lambda}}$},\ldots,\ooalign{$ A_{\lambda}$\cr
	$\phantom{A}^{(d)}$\cr
	$\phantom{A^{(d)}_{\lambda}}$}\subseteq\mathbb{T} $ such that if $ \alpha_1\in \ooalign{$ A_{\lambda}$\cr
	$\phantom{A}^{(1)}$\cr
	$\phantom{A^{(1)}_{\lambda}}$},\ldots,\alpha_d\in\ooalign{$ A_{\lambda}$\cr
	$\phantom{A}^{(d)}$\cr
	$\phantom{A^{(d)}_{\lambda}}$} $, then $ \vphantom{\Sigma}\smash{\Owidehat{\Sigma}} $ has gaps.
Furthermore, $ \ooalign{$ A_{\lambda}$\cr
	$\phantom{A}^{(k)}$\cr
	$\phantom{A^{(k)}_{\lambda}}$} $ can be chosen to be a subset of the Diophantine frequencies and lie within $ [0,\varepsilon] $ for any $ \varepsilon>0 $.
\end{thm}
\par
We say a few words on the proof of Theorem~\ref{thm: MT}.
Observe $ \vphantom{\Sigma}\smash{\Owidehat{\Sigma}} $ is the sum of Cantor spectra $ \Sigma_1+\cdots+\Sigma_d $, where henceforth $ \Sigma_k\coloneqq \Sigma_{\lambda_k,\alpha_k,\omega_k} $.
Indeed, the potential of $ \vphantom{H}\smash{\Owidehat{H}} $ is separable.
In general, the spectra of separable multidimensional operators are sums of the spectra of 1D operators.
This fact can be found within \cite{2011SquareFibonacciHamiltonian}.
See also another proof involving the convolution of density of states measures within \cite{2015SquareFibonacciHamiltonian}.
The Newhouse Gap Lemma may be utilized to guarantee that $ \Sigma_1+\cdots+\Sigma_d $ is an interval thereby establishing Theorem~\ref{thm: MT}.
The notion of thickness, which is a quantitative characterization of nonempty compact subsets $ K $ of $ \mathbb{R} $ often denoted $ \tau(K) $, was utilized by S.~Newhouse in the 1970s \cite{1970Newhouse,1974Newhouse,1979Newhouse} to prove the namesake Newhouse Gap Lemma.
See also for a short proof \texttt{section 4.2} on page 63 within \cite{1993PalisTakens}.
The definition of thickness can be found within \textsc{subsection~\ref{subsec: Thickness and Gap Lemmas}}.
We abridge the Newhouse Gap Lemma:
Let $ K_1 $ and $ K_2 $ be nonempty compact subsets of $ \mathbb{R} $.
Assume the maximal-gap-lengths of $ K_1 $ and $ K_2 $ are sufficiently small relative to the diameters of $ K_1 $ and $ K_2 $, and $ 1\le \tau(K_1)\cdot\tau(K_2) $.
Then $ K_1+K_2 $ is an interval.
S.~Astels in 1999 \cite{1999AstelsGapLemma1of2,1999AstelsGapLemma2of2} generalized the Newhouse Gap Lemma to obtain the following; see Theorem~\ref{thm: Thickness and Gap Lemmas(9)} for the unabridged version.
\begin{thm}[S.~Astels \cite{1999AstelsGapLemma1of2,1999AstelsGapLemma2of2}]
Let $ K_1,\ldots,K_d $ ($ d\ge 2 $) be nonempty compact subsets of $ \mathbb{R} $.
Assume the maximal-gap-lengths of $ K_1,\ldots,K_d $ are sufficiently small relative to the diameters of $ K_1,\ldots,K_d $, and $ 1\le \frac{\tau(K_1)}{\tau(K_1)+1}+\cdots+\frac{\tau(K_d)}{\tau(K_d)+1} $.
Then $ K_1+\cdots+K_d $ is an interval.
\end{thm}
\noindent%
Because $ \lambda_k $ is small, one can think of $ H_{\lambda_k,\alpha_k,\omega_k} $ as a perturbation of the discrete Laplacian whose spectrum is the interval $ [-2,2] $.
Specifically, $ \Sigma_k $ converges to $ [-2,2] $ in the Hausdorff metric as $ \lambda_k\to 0 $.
Also, the diameter of $ \Sigma_{k} $ converges to the diameter of $ [-2,2] $ as $ \lambda_k\to 0 $.
Indeed, $ |\mathrm{diam}\,\sigma(A)-\mathrm{diam}\,\sigma(B)|\le 2\,\mathrm{dist}_{\mathsf{Haus}}(\sigma(A),\sigma(B))\le 2\left\lVert A-B\right\rVert $ for bounded self-adjoint operators $ A $ and $ B $.
Therefore $ \Sigma_k $ has a small maximal-gap-length relative to the diameters of $ \Sigma_1,\ldots,\Sigma_d $.
As a result, to utilize the Newhouse Gap Lemma it is enough to establish that the thickness of the spectrum of the almost Mathieu operator approaches infinity as the coupling approaches zero.
\begin{thm}\label{thm: ITP}
The spectrum of the almost Mathieu operator $ H_{\lambda,\alpha,\omega} $ with irrational frequency fixed and assumed to satisfy a Diophantine condition
\[\textstyle \alpha\in\bigcup_{\substack{c>0\\t>1}}\bigcap_{\frac{p}{q}\in\mathbb{Q}}\left\{x\in\mathbb{R}:|qx-p|\ge \frac{c}{|q|^{t-1}}\right\}\mathrlap{{}\eqqcolon \mathsf{DC}} \]
has thickness approaching infinity as the coupling $ \lambda $ approaches zero i.e.\
\[\left(\forall \alpha\in \mathsf{DC}\right)[\textstyle \lim_{\lambda\to 0}\tau(\Sigma_{\lambda,\alpha,\omega})=+\infty]. \]
\end{thm}
\par
Theorem~\ref{thm: ITP} extends to 1D quasiperiodic operators whose potential is analytic and whose frequency is Diophantine; combine Theorem~\ref{thm: Integrated Density of States(3)} and Theorem~\ref{thm: Spectral Gaps(7)} and Lemma~\ref{lem: Proof of lem: Proof of thm: Main Theorem(2)(3)(1)}.
As a result, Theorem~\ref{thm: MT} extends to separable multidimensional discrete Schrödinger operators generated by 1D quasiperiodic operators whose potential is analytic and whose frequency is Diophantine.
\begin{thm}\label{thm: THEOREM}
For fixed Diophantine frequencies $ \alpha_1,\ldots,\alpha_d $, the spectrum is an interval for separable \mbox{$ d $-dimensional} discrete Schrödinger operators generated by 1D $ \alpha_k $-quasiperiodic analytic operators $ H_k=\Delta+V_k $ with $ \alpha_k $-dependent sufficiently small norms $ \lVert V_k\rVert $.
\end{thm}
Theorem~\ref{thm: THEOREM} is necessarily perturbative due to Theorem~\ref{thm: B}.
Specifically, the smallness of the norm depends on the Diophantine frequency.
We conclude this section with a few words on related mathematical results and questions.
The definitions of \textit{limit-periodic} and \textit{almost-periodic} potential can be found within \cite{2019MultiAPCantorSpectrum} and \cite{1982SimonReviewPaper}, respectively.
It suffices to say that the collection of all almost-periodic potentials is a broad class of potentials which contains the periodic and limit-periodic and quasiperiodic potentials.
The definitions of \textit{box-counting} and \textit{Hausdorff} dimension can be found within \cite{2014Falconer}.
A \textit{Cantorval} is a nonempty compact subset $ C $ of $ \mathbb{R} $ such that $ C $ has no isolated connected components and $ C $ has a dense interior; we mention for comparison that a \textit{Cantor set} is a nonempty compact subset $ C $ of $ \mathbb{R} $ such that $ C $ has no isolated points and $ C $ has no interior points.
\begin{enumerate}[(a)]
\item
We establish a single-interval-characterization for the spectra of separable multidimensional discrete Schrödinger operators generated by 1D quasiperiodic analytic operators $ H_k=\Delta+V_k $ with Diophantine-frequency-dependent sufficiently small norms $ \lVert V_k\rVert $.
The same characterization can be said about periodic $ V_k $ \cite{2018BetheSommerfeldConjecture} and cannot be said about limit-periodic $ V_k $ \cite{2019MultiAPCantorSpectrum}.
\item\label{HJ}
R.~Han and S.~Jitomirskaya \cite{2018BetheSommerfeldConjecture} proved that periodic (not necessarily separable) multidimensional discrete Schrödinger operators have interval spectra when the norm of the potential is small and at least one period is odd.
This is the discrete analog of the L.~Parnovski \cite{2008BetheSommerfeldConjecture} resolution of the continuous Bethe--Sommerfeld Conjecture.
Moreover, the resolution was extended to quasiperiodic (not necessarily separable) multidimensional continuous Schrödinger operators for almost-all frequencies \cite{2020BetheSommerfeldConjectureQPSOp}.
This paragraph is the only place where continuous Schrödinger operators are mentioned, and it is done so to motivate the following question.
Can one extend the \cite{2018BetheSommerfeldConjecture,2008BetheSommerfeldConjecture,2020BetheSommerfeldConjectureQPSOp} results to quasiperiodic (not necessarily separable) multidimensional discrete Schrödinger operators for almost-all frequencies?
Specifically, can one remove the separability condition from Theorem~\ref{thm: THEOREM}?
Moreover, can one remove the Diophantine condition from Theorem~\ref{thm: THEOREM}?
For example, can the frequencies satisfy a Liouvillian condition?
In a sense, this means that the irrational frequencies are well approximated by rational numbers.
\item\label{DFG}
D.~Damanik, J.~Fillman, and A.~Gorodetski \cite{2019MultiAPCantorSpectrum} proved that there exists a dense subset $ \mathcal{B} $ of 1D limit-periodic potentials such that $ \mathcal{B} $-type operators have Cantor spectra with zero (lower) box-counting dimension.
Furthermore, separable multidimensional discrete Schrödinger operators generated by $ \mathcal{B} $-type operators have Cantor spectra with zero (lower) box-counting dimension.
Is it true that the spectra of separable multidimensional discrete Schrödinger operators generated by 1D almost-periodic operators is either a finite union of disjoint intervals or a Cantor set or a Cantorval?
In view of Proposition~\ref{prop: PROP}, does there exist $ \lambda_1,\ldots,\lambda_d $ with at least two not equal to $ 1 $ and $ \alpha_1,\ldots,\alpha_d\in\mathbb{T}\setminus\mathbb{Q} $ such that $ \vphantom{\Sigma}\smash{\Owidehat{\Sigma}} $ is a Cantorval?
\item
D.~Damanik and A.~Gorodetski \cite{2011SquareFibonacciHamiltonian} proved that 1D Fibonacci Hamiltonians have Cantor spectra $ \Sigma_\lambda $ with Hausdorff dimension strictly between zero and one and with thickness strictly greater than zero, and $ \lim_{\lambda\to 0}\mathrm{dim}_{\mathsf{Haus}}(\Sigma_\lambda)=1 $ and $ \lim_{\lambda\to 0}\tau(\Sigma_\lambda)=+\infty $.
Consequently, they establish a single-interval-characterization for the spectra of separable multidimensional discrete Schrödinger operators generated by 1D Fibonacci Hamiltonians with sufficiently small couplings.
\item
For fixed $ s\in\mathbb{R}\setminus\left\{-1,0,1\right\} $ and for fixed Diophantine frequency in the 2-torus M.~Goldstein, W.~Schlag, and M.~Voda \cite{2019QPSOpLargeCoupling} proved that for sufficiently large couplings the 1D operators with 2-frequency quasiperiodic potential $ V_{\lambda,s,\alpha}:\mathbb{Z}\to\mathbb{R}:n\mapsto \lambda(\cos(2\pi n\alpha_1)+s\cos(2\pi n\alpha_2)) $ must have an interval spectrum.
This complements a result by J.~Bourgain \cite{2002SquareAMOII} establishing that for fixed small coupling the 1D operator with potential $ V_{\lambda,1,\alpha} $ has gaps in its spectrum for some (positive measure) Diophantine frequencies.
\item
The separable operator considered in Theorem~\ref{thm: THEOREM} but with added background potential (so the resulting operator is not necessarily separable) has been studied by J.~Bourgain and I.~Kachkovskiy \cite{2019SquareQPSOp}.
Similar but distinct quasiperiodic (not necessarily separable) operators have been studied by S.~Jitomirskaya, W.~Liu, and Y.~Shi \cite{2020MultiQPSOp}.
The \cite{2019SquareQPSOp,2020MultiQPSOp} results pertain to the spectral type but not the topological structure of the spectrum as a set.
\end{enumerate}
In section~\ref{sec: Preliminaries} we state the preliminaries.
In section~\ref{sec: Proof of thm: Main Theorem(2)} we prove the main theorem.
In section~\ref{sec: Proof of lem: Proof of thm: Main Theorem(2)(3)} we prove a lemma used in the proof of the main theorem.
\section{Preliminaries}\label{sec: Preliminaries}
\noindent%
Let $ A:\mathcal{H}\to\mathcal{H} $ be a bounded operator.
The \textit{spectrum} of $ A $ is denoted $ \sigma(A) $.
Note $ \sigma(A) $ is a nonempty compact subset of $ \mathbb{C} $.
Also note, if $ A $ is self-adjoint, then $ \sigma(A)\subseteq \mathbb{R} $.
\subsection{Potentials and Schrödinger Operators}\label{subsec: Potentials and Schrodinger Operators}
A ($ d $-dimensional lattice) \textit{potential} is a bounded real-valued function $ V $ on $ \mathbb{Z}^d $.
Moreover, $ V $ also denotes the bounded self-adjoint multiplication operator on $ \ell^2(\mathbb{Z}^d)$ defined by $ [V\psi](n)=V(n)\psi(n) $ for every $ \psi\in\ell^2(\mathbb{Z}^d),n\in\mathbb{Z}^d $.
The ($ d $-dimensional lattice) \textit{Schrödinger operator} with respect to a potential $ V:\mathbb{Z}^d\to\mathbb{R} $ is the bounded self-adjoint operator $ H $ on $\ell^2(\mathbb{Z}^d) $ defined by $ H=\Delta+V $, where $ \Delta=\Delta^{(d)} $ is the (\mbox{$ d $-dimensional} lattice) \textit{Laplacian} defined by $[\Delta\psi](n)=\sum_{m\in\left\{e_1,\ldots,e_d\right\}}\psi(n+m)+\psi(n-m) $ for every $ \psi\in\ell^2(\mathbb{Z}^d),n\in\mathbb{Z}^d $; $ \left\{e_1,\ldots,e_d\right\} $ is the standard basis.
\subsection{Separable Potentials and The Laplacian}\label{subsec: Separable Potentials and The Laplacian}
Let $ V:\mathbb{Z}^d\to\mathbb{R} $ be a potential.
Say $ V $ is \textit{separable} if there exist (sub)potentials $ V_1,\ldots,V_d:\mathbb{Z}\to \mathbb{R} $ such that $ V(n)=V_1(n_1)+\cdots+V_d(n_d) $ for every $ n\in\mathbb{Z}^d $.
The proof of the following theorem can be found within \cite{2011SquareFibonacciHamiltonian}.
\begin{thm}\label{thm: Separable Potentials and The Laplacian(2)}
Let $ V:\mathbb{Z}^d\to\mathbb{R}:n\mapsto V_1(n_1)+\cdots+V_d(n_d) $ be a separable potential.
Define $ H\coloneqq \Delta^{(d)}+V $.
For each $ k $, define $ H_k\coloneqq\Delta^{(1)}+V_k $.
Then
\[ \sigma(H)=\sigma(H_1)+\cdots+\sigma(H_d). \]
\end{thm}
\noindent%
By the spectral mapping theorem,
\[\textstyle \sigma(\Delta^{(1)})=\sigma(\ooalign{$ U_{1}$\cr
	$\phantom{U}^{*}$\cr
	$\phantom{U^{*}_{1}}$}+U_1)=\sigma(\Phi_{U_1}(z^*+z))=\{z^*+z:z\in\sigma(U_1)\}=\{z^*+z:|z|=1\}=[-2,2].\]
By Theorem~\ref{thm: Separable Potentials and The Laplacian(2)},
\[\textstyle\sigma(\Delta^{(d)})=\sigma(\Delta^{(1)})+\cdots+\sigma(\Delta^{(1)})\textnormal{ (\(d\) terms)}=[-2d,2d].\]
Here $ U_m $ is the unitary operator from $ \ell^2(\mathbb{Z}^d) $ to $ \ell^2(\mathbb{Z}^d) $ defined by $ [U_m\psi](n)=\psi(n-m) $ for every $ \psi\in \ell^2(\mathbb{Z}^d),n\in\mathbb{Z}^d $.
Also here, $ \Phi_A $ is the Borel functional calculus with respect to a bounded operator $ A $.
\subsection{Quasiperiodic Potentials and The Almost Mathieu Operator}\label{subsec: Quasiperiodic Potentials and The Almost Mathieu Operator}
Let $ V:\mathbb{Z}\to \mathbb{R} $ be a potential.
Let $ b $ be a positive integer.
Say $ V $ is ($ b $-frequency) \textit{quasiperiodic} if there exists $ v\in C(\mathbb{T}^b,\mathbb{R}) $ and there exist $ \alpha,\omega\in\mathbb{T}^b $ such that $ v $ is nonconstant and $ \left\{1,\alpha_1,\ldots,\alpha_b\right\} $ is independent over the rationals and $ V(n)=v(n\alpha+\omega) $ for every $ n\in\mathbb{Z} $.
The proof of the following theorem can be found within \cite{1982SimonReviewPaper}.
\begin{thm}\label{thm: Quasiperiodic Potentials and The Almost Mathieu Operator(2)}
Let $ V_\omega:\mathbb{Z}\to\mathbb{R}:n\mapsto v\big|_{\mathbb{T}^b}(n\alpha+\omega) $ be a quasiperiodic potential with parameter $ \omega $.
Let $ H_\omega $ be the Schrödinger operator.
%Define $ \Sigma_\om\defeq\sigma(H_\om) $, $ \Sigma^{\sss\cap}\defeq\bigcap_{\om\in\T^b}\Sigma_\om $, $ \Sigma^{\sss\cup}\defeq\bigcup_{\om\in\T^b}\Sigma_\om $.
Then $ \sigma(H_\omega)=\Sigma_{\omega}\eqqcolon \Sigma $ is independent of $ \omega $.
Furthermore, $ \Sigma $ is a nonempty compact subset of $ \mathbb{R} $ and $ \Sigma $ has no isolated points.
\end{thm}
\noindent%
The potential of the almost Mathieu operator $ V_{\lambda,\alpha,\omega}:\mathbb{Z}\to\mathbb{R}:n\mapsto2\lambda\cos(2\pi(n\alpha+\omega)) $ is $ 1 $-frequency quasiperiodic when $ \alpha $ is irrational.
By the resolution of the Ten Martini Problem within \cite{2009TheTenMartiniProblem}, the spectrum $ \Sigma_{\lambda,\alpha,\omega} $ of the almost Mathieu operator is a Cantor set when $ \alpha $ is irrational.
\subsection{Integrated Density of States}\label{subsec: Integrated Density of States}
Let $ V_{\omega}:\mathbb{Z}\to\mathbb{R}:n\mapsto v\big|_{\mathbb{T}^b}(n\alpha+\omega) $ be a quasiperiodic potential with parameter $ \omega $.
Let $ H_{\omega} $ be the Schrödinger operator.
The \textit{integrated density of states} (IDS) is
\begin{align*}
\mathfrak{N}:\mathbb{R}\to [0,1]:x
%&\ts\xmapsto{}\lim_{N\to\infty}\frac{1}{N}\,\#\regs{E\in (-\infty,x]:\regp{\e \psi\in \ell^2(\s{0,\ldots,N-1})}[H_{\om}\psi=E\psi]}\\
%&\ts\xmapsto{}\lim_{N\to\infty}\frac{1}{N}\,\mathrm{trace}(\chi_{\s{0,\ldots,N-1}}\1_{(-\infty,x]}(H_{\om})\chi_{\s{0,\ldots,N-1}})\\
%&\ts\xmapsto{}\lim_{N\to\infty}\frac{1}{N}\sum_{n\in\s{0,\ldots,N-1}}\regab{\de_n,\1_{(-\infty,x]}(H_{\om})\de_n}\\
%&\ts\xmapsto{}\lim_{N\to\infty}\frac{1}{N}\sum_{n\in\s{0,\ldots,N-1}}\regab{\sS{U}{n}{}\de_0,\1_{(-\infty,x]}(H_{\om})\sS{U}{n}{}\de_0}\\
%&\ts\xmapsto{}\lim_{N\to\infty}\frac{1}{N}\sum_{n\in\s{0,\ldots,N-1}}\regab{\de_0,\sS{U}{-n}{}\1_{(-\infty,x]}(H_{\om})\sS{U}{-n}{*}\de_0}\\
%&\ts\xmapsto{}\lim_{N\to\infty}\frac{1}{N}\sum_{n\in\s{0,\ldots,N-1}}\regab{\de_0,\1_{(-\infty,x]}(\sS{U}{-n}{}H_{\om}\sS{U}{-n}{*})\de_0}\\
%&\ts\xmapsto{}\lim_{N\to\infty}\frac{1}{N}\sum_{n\in\s{0,\ldots,N-1}}\regab{\de_0,\1_{(-\infty,x]}(H_{n\alpha+\om})\de_0}\\
&\mapsto \textstyle\int\langle\delta_0,\mathbbm{1}_{(-\infty,x]}(H_{\theta})\delta_0\rangle d\mu(\theta).
\end{align*}
Here $ \mu $ is the normalized Haar measure on $ \mathbb{T}^b $.
Note $ \mathfrak{N}\upharpoonright(-\infty,\inf \Sigma]\equiv 0 $ and $ \mathfrak{N}\upharpoonright[\sup \Sigma,+\infty)\equiv1 $.
The proof of Theorem~\ref{thm: Integrated Density of States(2)} can be found within \cite{1982SimonReviewPaper} and of Theorem~\ref{thm: Integrated Density of States(3)} can be found within \cite{2019HolderContinuity}.
\begin{thm}[Gap Labeling]\label{thm: Integrated Density of States(2)}
Let $ V:\mathbb{Z}\to\mathbb{R}:n\mapsto v\big|_{\mathbb{T}^b}(n\alpha+\omega) $ be a quasiperiodic potential.
Let $ H $ be the Schrödinger operator.
Let $ \mathfrak{N} $ be the IDS.
Define $ \Sigma\coloneqq\sigma(H) $.
Let $ \mathrm{Gap}_{\mathsf{b}}(\Sigma) $ be the collection of all bounded gaps of $ \Sigma $.
\begin{enumerate}[(i)]
\item
$ \mathfrak{N} $ is monotone and continuous.
\item
$ \left\{x\in\mathbb{R}:\left(\forall \varepsilon>0\right)[\mathfrak{N}(x+\varepsilon)-\mathfrak{N}(x-\varepsilon)>0]\right\}=\Sigma $.
\item
$ \left\{\mathfrak{N}(x):x\in U\in \mathrm{Gap}_{\mathsf{b}}(\Sigma)\right\}\subseteq \{\mathbf{n}\alpha-\left\lfloor\mathbf{n}\alpha\right\rfloor:\mathbf{n}\in\mathbb{Z}^b\}\setminus\left\{0\right\} $.
\item
$ \left\{\mathfrak{N}(x):x\in U\in \mathrm{Gap}_{\mathsf{b}}(\Sigma)\right\}\supseteq \{\mathbf{n}\alpha-\left\lfloor\mathbf{n}\alpha\right\rfloor:\mathbf{n}\in\mathbb{Z}^b\}\setminus\left\{0\right\} $ $ \implies $ $ \Sigma $ is a Cantor set.
\end{enumerate}
\end{thm}
\begin{thm}[Hölder Continuity]\label{thm: Integrated Density of States(3)}
Let $ V_{\lambda,v,\alpha,\omega}:\mathbb{Z}\to\mathbb{R}:n\mapsto \lambda v\big|_{\mathbb{T}^b}(n\alpha+\omega) $ be a quasiperiodic potential with parameters $ \lambda,v,\alpha,\omega $.
Let $ H_{\lambda,v,\alpha,\omega} $ be the Schrödinger operator.
Fix $ v,\alpha $.
Let $ \mathfrak{N}_{\lambda} $ be the IDS.
Assume the following.
\begin{enumerate}[(i)]
\item
$ \alpha $ satisfies a Diophantine condition:
\[\textstyle \smash{\alpha\in\bigcap_{\flsubstack{p\in\mathbb{Z}\\\smash{\mathbf{q}\in\mathbb{Z}^b\setminus\left\{\mathbf{0}\right\}}}}\left\{\mathbf{x}\in\mathbb{R}^b:|\mathbf{q}\mathbf{x}-2\pi p|\ge \frac{c}{\lVert\mathbf{q}\rVert^{t-1}}\right\}\mathrlap{{}\eqqcolon \vphantom{\mathsf{DC}}\smash{\Owidetilde{\mathsf{DC}}}{}^{b}(c,t)}} \]
for some $ c>0,t>b $.
\item
$ v $ satisfies a regularity condition: $ v $ is $ r $-differentiable on $ \mathbb{T}^b $ for some $ r\ge 550t $.
\end{enumerate}
There exist $ \lambda_0=\lambda_0(v,b,c,t,r)>0, C_{\mathsf{H}}=C_{\mathsf{H}}(b,c,t)>0 $ such that
\[
\smash{|\mathfrak{N}_\lambda(x)-\mathfrak{N}_\lambda(y)|\le C_{\mathsf{H}}|x-y|^{{1}/{2}}}
\]
for every $ \lambda_0\ge |\lambda|>0,x,y $.
\end{thm}
\subsection{Spectral Gaps}\label{subsec: Spectral Gaps}
Let $ V:\mathbb{Z}\to\mathbb{R}:n\mapsto v\big|_{\mathbb{T}^b}(n\alpha+\omega) $ be a quasiperiodic potential.
Let $ H $ be the Schrödinger operator.
Let $ \mathfrak{N} $ be the IDS.
Define $ \Sigma\coloneqq\sigma(H) $.
Let $ \mathrm{Gap}_{\mathsf{b}}(\Sigma) $ be the collection of all bounded gaps of $ \Sigma $.
Fix $ \mathbf{n}\in\mathbb{Z}^b\setminus\left\{\mathbf{0}\right\} $.
The $ \mathbf{n} $-th \textit{spectral gap} of $ \Sigma $ is either the element $ G^{(\mathbf{n})} $ in $ \mathrm{Gap}_{\mathsf{b}}(\Sigma) $ such that $ \mathfrak{N}\upharpoonright \smash{\overline{G^{(\mathbf{n})}}}\equiv \mathbf{n}\alpha-\lfloor\mathbf{n} \alpha\rfloor $ or the empty set when such an element $ G^{(\mathbf{n})} $ does not exist.
Moreover, $ (-\infty,\inf \Sigma)\cup(\sup \Sigma,+\infty) $ is called the $ \mathbf{0} $-th \textit{spectral gap} of $ \Sigma $.
The proof of the following theorem can be found within \cite{2017preprintGapEstimate}.
\begin{thm}[Gap Estimate]\label{thm: Spectral Gaps(7)}
Let $ V_{v,\alpha,\omega}:\mathbb{Z}\to\mathbb{R}:n\mapsto v\big|_{\mathbb{T}^b}(n\alpha+\omega) $ be a quasiperiodic potential with parameters $ v,\alpha,\omega $.
Let $ H_{v,\alpha,\omega} $ be the Schrödinger operator.
Fix $ \alpha $.
Fix $ r_0>r>0 $.
Define $ \Sigma_{v}\coloneqq\sigma(H_{v,\alpha,\omega}) $.
For each $ \mathbf{n}\ne \mathbf{0} $, let $ \ooalign{$ G_{v}$\cr
	$\phantom{G}^{(\mathbf{n})}$\cr
	$\phantom{G^{(\mathbf{n})}_{v}}$}=(\ooalign{$ E_{v}$\cr
	$\phantom{E}^{(\mathbf{n}){\scriptscriptstyle -}}$\cr
	$\phantom{E^{(\mathbf{n}){\scriptscriptstyle -}}_{v}}$},\ooalign{$ E_{v}$\cr
	$\phantom{E}^{(\mathbf{n}){\scriptscriptstyle +}}$\cr
	$\phantom{E^{(\mathbf{n}){\scriptscriptstyle +}}_{v}}$}) $ be the $ \mathbf{n} $-th spectral gap of $ \Sigma_{v} $.
Assume the following.
\begin{enumerate}[(i)]
\item
$ \alpha $ satisfies a Diophantine condition:
\[\textstyle \smash{\alpha\in\bigcap_{\flsubstack{p\in\mathbb{Z}\\\smash{\mathbf{q}\in\mathbb{Z}^b\setminus\left\{\mathbf{0}\right\}}}}\left\{\mathbf{x}\in\mathbb{R}^b:|\mathbf{q}\mathbf{x}-p|\ge \frac{c}{\lVert\mathbf{q}\rVert^{t-1}}\right\}\mathrlap{{}\eqqcolon \mathsf{DC}{}^{b}(c,t)}} \]
for some $ c>0,t>b $.
\item
$ v $ satisfies a regularity condition: $ v $ is analytic on $ (\mathbb{T}+i(-r_0,r_0))^b $.
\end{enumerate}
There exists $ \varepsilon=\varepsilon(b,c,t,r_0,r)>0 $ such that
\[
\textstyle\smash{|\ooalign{$ E_{v}$\cr
	$\phantom{E}^{(\mathbf{n}){\scriptscriptstyle +}}$\cr
	$\phantom{E^{(\mathbf{n}){\scriptscriptstyle +}}_{v}}$}-\ooalign{$ E_{v}$\cr
	$\phantom{E}^{(\mathbf{n}){\scriptscriptstyle -}}$\cr
	$\phantom{E^{(\mathbf{n}){\scriptscriptstyle -}}_{v}}$}|\le (|v|_{r_0})^{2/3}e^{-2\pi r\lVert\mathbf{n}\rVert}}
\]
for every $ \varepsilon\ge |v|_{r_0}>0,\mathbf{n}\ne \mathbf{0} $.
\end{thm}
\subsection{Thickness and Gap Lemmas}\label{subsec: Thickness and Gap Lemmas}
Let $ K $ be a nonempty compact subset of $ \mathbb{R} $.
Henceforth $ K^{\scriptscriptstyle -}\coloneqq \inf K $ and $ K^{\scriptscriptstyle +}\coloneqq \sup K $.
Let $ \mathrm{Gap}_{\mathsf{b}}(K) $ be the collection of all bounded gaps of $ K $.
Let $ U $ be a bounded gap of $ K $.
The \textit{left-plank} of $ U $, denoted $ \pi_{\scriptscriptstyle -}(K,U) $, is the length-maximal interval $ [a,b] $ contained in $ [K^{\scriptscriptstyle -},K^{\scriptscriptstyle +}] $ \mbox{such that} $ b\in\partial U $ and for each $ V\in\mathrm{Gap}_{\mathsf{b}}(K) $, if $ V\cap [a,b]\ne\varnothing $, then $ \mathrm{length}(V)< \mathrm{length}(U) $.
The \textit{right-plank} of $ U $, denoted $ \pi_{\scriptscriptstyle +}(K,U) $, is the length-maximal interval $ [a,b] $ contained in $ [K^{\scriptscriptstyle -},K^{\scriptscriptstyle +}] $ such that $ a\in\partial U $ and for each $ V\in\mathrm{Gap}_{\mathsf{b}}(K) $, if $ V\cap [a,b]\ne\varnothing $, then $ \mathrm{length}(V)< \mathrm{length}(U) $.
The (local) \textit{thickness} of $ U $ is
\[
\inf\nolimits_{{\bullet}\in\left\{-,+\right\}}\frac{\mathrm{length}(\pi_{\scriptscriptstyle \bullet}(K,U))}{\mathrm{length}(U)}\mathrlap{{}\eqqcolon \tau(K,U).}
\]
The \textit{thickness} of $ K $ is
\[\smash{\inf\nolimits_{U\in\mathrm{Gap}_{\mathsf{b}}(K)}\tau(K,U)\mathrlap{{}\eqqcolon \tau(K).}}\]
Note $ \tau(K)=0 $ when $ K $ has an isolated point, $ \tau(K)=1 $ when $ K $ is the middle-thirds Cantor set, and $ \tau(K)=+\infty $ when $ K $ is an interval.
The proof of Theorem~\ref{thm: Thickness and Gap Lemmas(7)} can be found within \cite{2011SquareFibonacciHamiltonian} and of Theorem~\ref{thm: Thickness and Gap Lemmas(9)} can be found within \cite{1999AstelsGapLemma1of2,1999AstelsGapLemma2of2}.
The former theorem is the classical Gap Lemma, but we require the latter theorem which is a generalized version.
\begin{thm}[Gap Lemma]\label{thm: Thickness and Gap Lemmas(7)}
Let $ K_1 $ and $ K_2 $ be nonempty compact subsets of $ \mathbb{R} $.
For each nonempty compact subset $ K $ of $ \mathbb{R} $, let $ \mathrm{Gap}_{\mathsf{b}}(K) $ be the collection of all bounded gaps of $ K $.
For each nonempty compact subset $ K $ of $ \mathbb{R} $, define $\Gamma(K)\coloneqq\sup\{\mathrm{length}(U):U\in\mathrm{Gap}_{\mathsf{b}}(K)\} $.
Assume $ [K_{1}^{\scriptscriptstyle -},K_{1}^{\scriptscriptstyle +}]\cap [K_{2}^{\scriptscriptstyle -},K_{2}^{\scriptscriptstyle +}]\ne\varnothing $, $ \Gamma(K_2)\le \mathrm{diam}(K_1) $, $ \Gamma(K_1)\le \mathrm{diam}(K_2) $, and $ 1\le \tau(K_1)\cdot\tau(K_2) $.
Then $ K_1 \cap K_2\ne\varnothing $.
Furthermore, $ K_1+K_2=[K_{1}^{\scriptscriptstyle -}+K_{2}^{\scriptscriptstyle -},K_{1}^{\scriptscriptstyle +}+K_{2}^{\scriptscriptstyle +}] $.
\end{thm}
\begin{thm}[Gap Lemma]\label{thm: Thickness and Gap Lemmas(9)}
Let $ K_1,\ldots,K_d $ ($ d\ge 2 $) be nonempty compact subsets of $ \mathbb{R} $.
For each nonempty compact subset $ K $ of $ \mathbb{R} $, let $ \mathrm{Gap}_{\mathsf{b}}(K) $ be the collection of all bounded gaps of $ K $.
For each nonempty compact subset $ K $ of $ \mathbb{R} $, define $\Gamma(K)\coloneqq\sup\{\mathrm{length}(U):U\in\mathrm{Gap}_{\mathsf{b}}(K)\} $.
Assume
\[ 
\begin{cases}
\left(\forall i:2\le i\le d\right)\left(\forall j:1\le j\le i-1\right)[\Gamma(K_j)\le \mathrm{diam}(K_i)],\\
\left(\forall i:2\le i\le d\right)[\Gamma(K_i)\le \mathrm{diam}(K_1)+\cdots+\mathrm{diam}(K_{i-1})].
\end{cases}
\]
If $ 1\le \frac{\tau(K_1)}{\tau(K_1)+1}+\cdots+\frac{\tau(K_d)}{\tau(K_d)+1} $, then $ \tau(K_1+\cdots+K_d)=+\infty $ and
\[ K_1+\cdots+K_d=[K_{1}^{\scriptscriptstyle -}+\cdots+K_{d}^{\scriptscriptstyle -},K_{1}^{\scriptscriptstyle +}+\cdots+K_{d}^{\scriptscriptstyle +}]. \]
If $ \frac{\tau(K_1)}{\tau(K_1)+1}+\cdots+\frac{\tau(K_d)}{\tau(K_d)+1}<1 $, then $ \tau(K_1+\cdots+K_d)\ge \frac{\frac{\tau(K_1)}{\tau(K_1)+1}+\cdots+\frac{\tau(K_d)}{\tau(K_d)+1}}{1-(\frac{\tau(K_1)}{\tau(K_1)+1}+\cdots+\frac{\tau(K_d)}{\tau(K_d)+1})} $.
\end{thm}
\section{Proof of Main Theorem}\label{sec: Proof of thm: Main Theorem(2)}
\noindent%
\begin{lem}\label{lem: Proof of thm: Main Theorem(2)(2)}
Let $ A:\mathcal{H}\to\mathcal{H} $ and $ B:\mathcal{H}\to\mathcal{H} $ be bounded self-adjoint operators.
Then
\[
|\mathrm{diam}\,\sigma(A)-\mathrm{diam}\,\sigma(B)|\le 2\,\mathrm{dist}_{\mathsf{Haus}}(\sigma(A),\sigma(B))\le 2\left\lVert A-B\right\rVert.
\]
\end{lem}
\begin{proof}[\textcolor{black}{Proof.}]\renewcommand{\qedsymbol}{\textcolor{black}{\openbox}}%
Define $ \xi\coloneqq \min\sigma(A)-\min\sigma(B) $, $ \zeta\coloneqq \max\sigma(A)-\max\sigma(B) $.
Define $ s\curlyvee t\coloneqq \sup\left\{s,t\right\} $.
Because
\begin{align*}
|\xi|\curlyvee|\zeta|
&\textstyle\le \sup_{x\in\sigma(A)}\mathrm{dist}(x,\sigma(B))\curlyvee\sup_{x\in\sigma(B)}\mathrm{dist}(x,\sigma(A))\mathrlap{{}\eqqcolon \mathrm{dist}_{\mathsf{Haus}}(\sigma(A),\sigma(B))}\\
&\textstyle= \sup_{x\in\sigma(A)}\inf_{\left\lVert\psi\right\rVert=1}\left\lVert(B-xI)\psi\right\rVert\curlyvee\sup_{x\in\sigma(B)}\inf_{\left\lVert\psi\right\rVert=1}\left\lVert(A-xI)\psi\right\rVert\\
&\textstyle\le \sup_{x\in\sigma(A)}\left\lVert\smash{B-A}\right\rVert\curlyvee\sup_{x\in\sigma(B)}\left\lVert\smash{A-B}\right\rVert\\
&\textstyle= \left\lVert A-B\right\rVert,
\end{align*}
$ |\mathrm{diam}\,\sigma(A)-\mathrm{diam}\,\sigma(B)|=|\zeta-\xi|\le |\zeta|+|\xi|\le 2\,\mathrm{dist}_{\mathsf{Haus}}(\sigma(A),\sigma(B)) \le 2\left\lVert A-B\right\rVert $.
\end{proof}
\begin{lem}\label{lem: Proof of thm: Main Theorem(2)(3)}
Let $ H_{\lambda,\alpha,\omega}=\Delta+V_{\lambda,\alpha,\omega} $ be the almost Mathieu operator.
Fix $ \alpha\in\mathbb{T}\setminus\mathbb{Q} $.
Define $ \Sigma_\lambda\coloneqq\sigma(H_{\lambda,\alpha,\omega}) $.
Assume $ \alpha $ satisfies a Diophantine condition:
\[\textstyle \alpha\in\bigcup_{\substack{c>0\\t>1}}\bigcap_{\frac{p}{q}\in\mathbb{Q}}\left\{x\in\mathbb{R}:|qx-p|\ge \frac{c}{|q|^{t-1}}\right\}. \]
Then $ \lim_{\lambda\to 0}\tau(\Sigma_\lambda)=+\infty $.
\end{lem}
\noindent%
The proof of Lemma~\ref{lem: Proof of thm: Main Theorem(2)(3)} can be found within \textsc{section~\ref{sec: Proof of lem: Proof of thm: Main Theorem(2)(3)}}.
\begin{thm}[Main Theorem]\label{thm: Proof of thm: Main Theorem(2)(5)}
Fix $ \alpha_1,\ldots,\alpha_d,\omega_1,\ldots,\omega_d\in [0,1]\eqqcolon \mathbb{T} $.
Let $ \vphantom{H}\smash{\Owidehat{H}} $ be the bounded self-adjoint operator $ \vphantom{H}\smash{\Owidehat{H}}:\ell^2(\mathbb{Z}^d)\to \ell^2(\mathbb{Z}^d) $ defined by
\[
\textstyle[\vphantom{H}\smash{\Owidehat{H}}\psi](n)=\big(\sum_{m\in \left\{e_1,\ldots,e_d\right\}}\psi(n+m)+\psi(n-m)\big)+\big(\sum_{k\in\left\{1,\ldots,d\right\}}2\lambda_k\cos(2\pi(n_k\alpha_k+\omega_k))\big)\psi(n)
\]
for every $ \psi\in\ell^2(\mathbb{Z}^d),n\in\mathbb{Z}^d $; $ \left\{e_1,\ldots,e_d\right\} $ is the standard basis.
\par
Let $ H_{\lambda,\alpha,\omega}=\Delta+V_{\lambda,\alpha,\omega} $ be the almost Mathieu operator.
For each $ k $, define $ \Sigma_k\coloneqq\sigma(H_{\lambda_k,\alpha_k,\omega_k}) $.
Assume $ \alpha_1,\ldots,\alpha_d $ ($ d\ge 2 $) are irrational and satisfy a Diophantine condition:
\[\textstyle \alpha\in\bigcup_{\substack{c>0\\t>1}}\bigcap_{\frac{p}{q}\in\mathbb{Q}}\left\{x\in\mathbb{R}:|qx-p|\ge \frac{c}{|q|^{t-1}}\right\}. \]
There exists $ \varepsilon=\varepsilon(\alpha_1,\ldots,\alpha_d)>0 $ such that if $ 0<|\lambda_1|,\ldots,|\lambda_d|<\varepsilon $, then $ \sigma(\vphantom{H}\smash{\Owidehat{H}}) $, which is a sum of Cantor spectra $ \Sigma_1+\cdots+\Sigma_d $, is an interval.
\end{thm}
\begin{proof}[\textcolor{black}{Proof.}]\renewcommand{\qedsymbol}{\textcolor{black}{$ \blacksquare $}}%
For each nonempty compact subset $ K $ of $ \mathbb{R} $, let $ \mathrm{Gap}_{\mathsf{b}}(K) $ be the collection of all bounded gaps of $ K $.
For each nonempty compact subset $ K $ of $ \mathbb{R} $, define $\Gamma(K)\coloneqq\sup\{\mathrm{length}(U):U\in\mathrm{Gap}_{\mathsf{b}}(K)\} $.
By Lemma~\ref{lem: Proof of thm: Main Theorem(2)(2)}, $ \lim_{\lambda\to 0}\mathrm{dist}_{\mathsf{Haus}}(\sigma(H_{\lambda,\alpha,\omega}),[-2,2])=0 $ and $ \lim_{\lambda\to 0}|\mathrm{diam}\,\sigma(H_{\lambda,\alpha,\omega})-4|=0 $.
By Lemma~\ref{lem: Proof of thm: Main Theorem(2)(3)}, $ \lim_{\lambda\to 0}\tau(\sigma(H_{\lambda,\alpha_k,\omega}))=+\infty $.
As a result, there exists $ \varepsilon=\varepsilon(\alpha_1,\ldots,\alpha_d)>0 $ such that if $ 0<|\lambda_1|,\ldots,|\lambda_d|<\varepsilon $, then
\[
\begin{cases}
\left(\forall i:2\le i\le d\right)\left(\forall j:1\le j\le i-1\right)[\Gamma(\Sigma_j)\le \mathrm{diam}(\Sigma_i)],\\
\left(\forall i:2\le i\le d\right)[\Gamma(\Sigma_i)\le \mathrm{diam}(\Sigma_1)+\cdots+\mathrm{diam}(\Sigma_{i-1})],\\
1\le \frac{\tau(\Sigma_1)}{\tau(\Sigma_1)+1}+\cdots+\frac{\tau(\Sigma_d)}{\tau(\Sigma_d)+1}.
\end{cases}
\]
By Theorem~\ref{thm: Thickness and Gap Lemmas(9)}, $ \Sigma_1+\cdots+\Sigma_d $ is an interval.
By Theorem~\ref{thm: Separable Potentials and The Laplacian(2)}, $ \sigma(\vphantom{H}\smash{\Owidehat{H}})=\Sigma_1+\cdots+\Sigma_d $.
\end{proof}
\section{Proof of Lemma~\ref{lem: Proof of thm: Main Theorem(2)(3)}}\label{sec: Proof of lem: Proof of thm: Main Theorem(2)(3)}
\noindent%
\begin{lem}\label{lem: Proof of lem: Proof of thm: Main Theorem(2)(3)(1)}
Let $ V_{\lambda,\alpha,\omega}:\mathbb{Z}\to\mathbb{R}:n\mapsto \lambda v\big|_{\mathbb{T}^b}(n\alpha+\omega) $ be a quasiperiodic potential with parameters $ \lambda,\alpha,\omega $.
Let $ H_{\lambda,\alpha,\omega} $ be the Schrödinger operator.
Fix $ \alpha $.
Let $ \mathfrak{N}_{\lambda} $ be the IDS.
Define $ \Sigma_{\lambda}\coloneqq\sigma(H_{\lambda,\alpha,\omega}) $.
For each $ \mathbf{n}\ne \mathbf{0} $, let $ \ooalign{$ G_{\lambda}$\cr
	$\phantom{G}^{(\mathbf{n})}$\cr
	$\phantom{G^{(\mathbf{n})}_{\lambda}}$}=(\ooalign{$ E_{\lambda}$\cr
	$\phantom{E}^{(\mathbf{n}){\scriptscriptstyle -}}$\cr
	$\phantom{E^{(\mathbf{n}){\scriptscriptstyle -}}_{\lambda}}$},\ooalign{$ E_{\lambda}$\cr
	$\phantom{E}^{(\mathbf{n}){\scriptscriptstyle +}}$\cr
	$\phantom{E^{(\mathbf{n}){\scriptscriptstyle +}}_{\lambda}}$}) $ be the $ \mathbf{n} $-th spectral gap of $ \Sigma_{\lambda} $.
Assume the following.
\begin{enumerate}[(i)]
\item
$ \alpha $ satisfies a Diophantine condition:
\[\textstyle \alpha\in\bigcap_{\flsubstack{p\in\mathbb{Z}\\\smash{\mathbf{q}\in\mathbb{Z}^b\setminus\left\{\mathbf{0}\right\}}}}\left\{\mathbf{x}\in\mathbb{R}^b:|\mathbf{q}\mathbf{x}-p|\ge \frac{c}{\lVert\mathbf{q}\rVert^{t-1}}\right\}\mathrlap{{}\eqqcolon \mathsf{DC}{}^{b}(c,t)} \]
for some $ c>0,t>b $.
\item
There exist $ \lambda_0=\lambda_0(b,c,t)>0, C_{\mathsf{H}}=C_{\mathsf{H}}(b,c,t)>0,1\ge h=h(b,c,t)>0 $ such that
\[
|\mathfrak{N}_\lambda(x)-\mathfrak{N}_\lambda(y)|\le C_{\mathsf{H}}|x-y|^{h}
\]
for every $ \lambda_0\ge |\lambda|>0,x,y $.
\item
There exist $ \lambda_1=\lambda_1(b,c,t)>0, C(\lambda)=C(\lambda,b,c,t)>0, C_{\mathsf{E}}=C_{\mathsf{E}}(b,c,t)>0 $ such that
\[
|\ooalign{$ E_{\lambda}$\cr
	$\phantom{E}^{(\mathbf{n}){\scriptscriptstyle +}}$\cr
	$\phantom{E^{(\mathbf{n}){\scriptscriptstyle +}}_{\lambda}}$}-\ooalign{$ E_{\lambda}$\cr
	$\phantom{E}^{(\mathbf{n}){\scriptscriptstyle -}}$\cr
	$\phantom{E^{(\mathbf{n}){\scriptscriptstyle -}}_{\lambda}}$}|\le C(\lambda)e^{-C_{\mathsf{E}}\lVert\mathbf{n}\rVert}
\]
for every $ \lambda_1\ge |\lambda|>0,\mathbf{n}\ne \mathbf{0} $ and $ \lim_{\lambda\to 0}C(\lambda)=0 $.
\end{enumerate}
Then $\textstyle \lim_{\lambda\to 0}\tau(\Sigma_{\lambda})=+\infty $.
\end{lem}
\begin{proof}[\textcolor{black}{Proof.}]\renewcommand{\qedsymbol}{\textcolor{black}{\openbox}}%
For each $ \mathbf{n}\ne\mathbf{0} $, let $ \pi(\Sigma_\lambda,\ooalign{$ G_{\lambda}$\cr
	$\phantom{G}^{(\mathbf{n})}$\cr
	$\phantom{G^{(\mathbf{n})}_{\lambda}}$}) $ be a length-minimal plank of $ \ooalign{$ G_{\lambda}$\cr
	$\phantom{G}^{(\mathbf{n})}$\cr
	$\phantom{G^{(\mathbf{n})}_{\lambda}}$}\ne\varnothing $.
Without loss of generality, assume
\[ \pi(\Sigma_\lambda,\ooalign{$ G_{\lambda}$\cr
	$\phantom{G}^{(\mathbf{n})}$\cr
	$\phantom{G^{(\mathbf{n})}_{\lambda}}$})=\pi_{\scriptscriptstyle +}(\Sigma_\lambda,\ooalign{$ G_{\lambda}$\cr
	$\phantom{G}^{(\mathbf{n})}$\cr
	$\phantom{G^{(\mathbf{n})}_{\lambda}}$})=[\ooalign{$ E_{\lambda}$\cr
	$\phantom{E}^{(\mathbf{n}){\scriptscriptstyle +}}$\cr
	$\phantom{E^{(\mathbf{n}){\scriptscriptstyle +}}_{\lambda}}$},\ooalign{$ F_{\lambda}$\cr
	$\phantom{F}^{(\mathbf{n})}$\cr
	$\phantom{F^{(\mathbf{n})}_{\lambda}}$}] \]
for some $ \ooalign{$ F_{\lambda}$\cr
	$\phantom{F}^{(\mathbf{n})}$\cr
	$\phantom{F^{(\mathbf{n})}_{\lambda}}$}\in \Sigma_\lambda $.
Observe
\begin{align*}
\tau(\Sigma_{\lambda},\ooalign{$ G_{\lambda}$\cr
	$\phantom{G}^{(\mathbf{n})}$\cr
	$\phantom{G^{(\mathbf{n})}_{\lambda}}$})=\frac{\mathrm{length}(\pi(\Sigma_\lambda,\ooalign{$ G_{\lambda}$\cr
	$\phantom{G}^{(\mathbf{n})}$\cr
	$\phantom{G^{(\mathbf{n})}_{\lambda}}$}))}{\mathrm{length}(\ooalign{$ G_{\lambda}$\cr
	$\phantom{G}^{(\mathbf{n})}$\cr
	$\phantom{G^{(\mathbf{n})}_{\lambda}}$})}
= \frac{|\ooalign{$ F_{\lambda}$\cr
	$\phantom{F}^{(\mathbf{n})}$\cr
	$\phantom{F^{(\mathbf{n})}_{\lambda}}$}-\ooalign{$ E_{\lambda}$\cr
	$\phantom{E}^{(\mathbf{n}){\scriptscriptstyle +}}$\cr
	$\phantom{E^{(\mathbf{n}){\scriptscriptstyle +}}_{\lambda}}$}|}{|\ooalign{$ E_{\lambda}$\cr
	$\phantom{E}^{(\mathbf{n}){\scriptscriptstyle +}}$\cr
	$\phantom{E^{(\mathbf{n}){\scriptscriptstyle +}}_{\lambda}}$}-\ooalign{$ E_{\lambda}$\cr
	$\phantom{E}^{(\mathbf{n}){\scriptscriptstyle -}}$\cr
	$\phantom{E^{(\mathbf{n}){\scriptscriptstyle -}}_{\lambda}}$}|}.
\end{align*}
Because $ \Sigma_\lambda $ has no isolated points,
\[ \ooalign{$ E_{\lambda}$\cr
	$\phantom{E}^{(\mathbf{n}){\scriptscriptstyle +}}$\cr
	$\phantom{E^{(\mathbf{n}){\scriptscriptstyle +}}_{\lambda}}$}<\ooalign{$ F_{\lambda}$\cr
	$\phantom{F}^{(\mathbf{n})}$\cr
	$\phantom{F^{(\mathbf{n})}_{\lambda}}$}\le \sup\Sigma_{\lambda}. \]
Temporarily, assume $ \ooalign{$ F_{\lambda}$\cr
	$\phantom{F}^{(\mathbf{n})}$\cr
	$\phantom{F^{(\mathbf{n})}_{\lambda}}$}=\sup\Sigma_{\lambda} $.
Therefore
\[ |\mathfrak{N}_\lambda(\ooalign{$ F_{\lambda}$\cr
	$\phantom{F}^{(\mathbf{n})}$\cr
	$\phantom{F^{(\mathbf{n})}_{\lambda}}$})-\mathfrak{N}_\lambda(\ooalign{$ E_{\lambda}$\cr
	$\phantom{E}^{(\mathbf{n}){\scriptscriptstyle +}}$\cr
	$\phantom{E^{(\mathbf{n}){\scriptscriptstyle +}}_{\lambda}}$})|=|1-(\mathbf{n}\alpha-\lfloor\mathbf{n}\alpha\rfloor)|=|\mathbf{n} \alpha -\lfloor\mathbf{n} \alpha\rfloor-1|. \]
By \textit{(i--iii)}, for sufficiently small $ \lambda $,
\begin{align*}
\frac{|\ooalign{$ F_{\lambda}$\cr
	$\phantom{F}^{(\mathbf{n})}$\cr
	$\phantom{F^{(\mathbf{n})}_{\lambda}}$}-\ooalign{$ E_{\lambda}$\cr
	$\phantom{E}^{(\mathbf{n}){\scriptscriptstyle +}}$\cr
	$\phantom{E^{(\mathbf{n}){\scriptscriptstyle +}}_{\lambda}}$}|}{|\ooalign{$ E_{\lambda}$\cr
	$\phantom{E}^{(\mathbf{n}){\scriptscriptstyle +}}$\cr
	$\phantom{E^{(\mathbf{n}){\scriptscriptstyle +}}_{\lambda}}$}-\ooalign{$ E_{\lambda}$\cr
	$\phantom{E}^{(\mathbf{n}){\scriptscriptstyle -}}$\cr
	$\phantom{E^{(\mathbf{n}){\scriptscriptstyle -}}_{\lambda}}$}|}
&\ge \frac{(\frac{1}{C_{\mathsf{H}}})^{\frac{1}{h}}|\mathfrak{N}_\lambda(\ooalign{$ F_{\lambda}$\cr
	$\phantom{F}^{(\mathbf{n})}$\cr
	$\phantom{F^{(\mathbf{n})}_{\lambda}}$})-\mathfrak{N}_\lambda(\ooalign{$ E_{\lambda}$\cr
	$\phantom{E}^{(\mathbf{n}){\scriptscriptstyle +}}$\cr
	$\phantom{E^{(\mathbf{n}){\scriptscriptstyle +}}_{\lambda}}$})|^{\frac{1}{h}}}{|\ooalign{$ E_{\lambda}$\cr
	$\phantom{E}^{(\mathbf{n}){\scriptscriptstyle +}}$\cr
	$\phantom{E^{(\mathbf{n}){\scriptscriptstyle +}}_{\lambda}}$}-\ooalign{$ E_{\lambda}$\cr
	$\phantom{E}^{(\mathbf{n}){\scriptscriptstyle -}}$\cr
	$\phantom{E^{(\mathbf{n}){\scriptscriptstyle -}}_{\lambda}}$}|}
\ge \frac{(\frac{1}{C_{\mathsf{H}}})^{\frac{1}{h}}|\frac{c}{\lVert\mathbf{n}\rVert^{t-1}}|^{\frac{1}{h}}}{|\ooalign{$ E_{\lambda}$\cr
	$\phantom{E}^{(\mathbf{n}){\scriptscriptstyle +}}$\cr
	$\phantom{E^{(\mathbf{n}){\scriptscriptstyle +}}_{\lambda}}$}-\ooalign{$ E_{\lambda}$\cr
	$\phantom{E}^{(\mathbf{n}){\scriptscriptstyle -}}$\cr
	$\phantom{E^{(\mathbf{n}){\scriptscriptstyle -}}_{\lambda}}$}|}
\ge \frac{(\frac{c}{C_{\mathsf{H}}})^{{1}/{h}}}{C(\lambda)e^{-C_{\mathsf{E}}\lVert\mathbf{n}\rVert}\lVert\mathbf{n}\rVert^{{(t-1)}/{h}}}.
\end{align*}
Temporarily, assume $ \ooalign{$ F_{\lambda}$\cr
	$\phantom{F}^{(\mathbf{n})}$\cr
	$\phantom{F^{(\mathbf{n})}_{\lambda}}$}< \sup\Sigma_{\lambda} $.
Therefore
\[ |\mathfrak{N}_\lambda(\ooalign{$ F_{\lambda}$\cr
	$\phantom{F}^{(\mathbf{n})}$\cr
	$\phantom{F^{(\mathbf{n})}_{\lambda}}$})-\mathfrak{N}_\lambda(\ooalign{$ E_{\lambda}$\cr
	$\phantom{E}^{(\mathbf{n}){\scriptscriptstyle +}}$\cr
	$\phantom{E^{(\mathbf{n}){\scriptscriptstyle +}}_{\lambda}}$})|=|(\mathbf{m}_\mathbf{n}\alpha-\lfloor\mathbf{m}_\mathbf{n}\alpha\rfloor)-(\mathbf{n}\alpha-\lfloor\mathbf{n}\alpha\rfloor)|=|(\mathbf{m}_\mathbf{n}-\mathbf{n}) \alpha -(\lfloor\mathbf{m}_\mathbf{n}\alpha\rfloor-\lfloor\mathbf{n} \alpha\rfloor)| \]
for some $ \mathbf{m}_\mathbf{n} $ such that $ \ooalign{$ G_{\lambda}$\cr
	$\phantom{G}^{(\mathbf{m}_\mathbf{n})}$\cr
	$\phantom{G^{(\mathbf{m}_\mathbf{n})}_{\lambda}}$}\ne\varnothing $ and $ \mathrm{length}(\ooalign{$ G_{\lambda}$\cr
	$\phantom{G}^{(\mathbf{n})}$\cr
	$\phantom{G^{(\mathbf{n})}_{\lambda}}$})<\mathrm{length}(\ooalign{$ G_{\lambda}$\cr
	$\phantom{G}^{(\mathbf{m}_\mathbf{n})}$\cr
	$\phantom{G^{(\mathbf{m}_\mathbf{n})}_{\lambda}}$}) $ and $ \ooalign{$ F_{\lambda}$\cr
	$\phantom{F}^{(\mathbf{n})}$\cr
	$\phantom{F^{(\mathbf{n})}_{\lambda}}$}=\ooalign{$ E_{\lambda}$\cr
	$\phantom{E}^{(\mathbf{m}_\mathbf{n}){\scriptscriptstyle -}}$\cr
	$\phantom{E^{(\mathbf{m}_\mathbf{n}){\scriptscriptstyle -}}_{\lambda}}$} $ and
\[\textstyle \lVert\mathbf{m}_\mathbf{n}\rVert\le \frac{1}{-C_{\mathsf{E}}}\log (\frac{1}{C(\lambda)}\mathrm{length}(\ooalign{$ G_{\lambda}$\cr
	$\phantom{G}^{(\mathbf{n})}$\cr
	$\phantom{G^{(\mathbf{n})}_{\lambda}}$}))\mathrlap{{}\eqqcolon\kappa_\lambda(\mathbf{n}),} \]
where $ \mathrm{length}(\ooalign{$ G_{\lambda}$\cr
	$\phantom{G}^{(\mathbf{n})}$\cr
	$\phantom{G^{(\mathbf{n})}_{\lambda}}$})=C(\lambda)e^{-C_{\mathsf{E}}\kappa_\lambda(\mathbf{n})} $.
Because $ \lVert\mathbf{n}\rVert\le \kappa_\lambda(\mathbf{n}) $, for sufficiently small $ \lambda $,
\begin{align*}
\frac{|\ooalign{$ F_{\lambda}$\cr
	$\phantom{F}^{(\mathbf{n})}$\cr
	$\phantom{F^{(\mathbf{n})}_{\lambda}}$}-\ooalign{$ E_{\lambda}$\cr
	$\phantom{E}^{(\mathbf{n}){\scriptscriptstyle +}}$\cr
	$\phantom{E^{(\mathbf{n}){\scriptscriptstyle +}}_{\lambda}}$}|}{|\ooalign{$ E_{\lambda}$\cr
	$\phantom{E}^{(\mathbf{n}){\scriptscriptstyle +}}$\cr
	$\phantom{E^{(\mathbf{n}){\scriptscriptstyle +}}_{\lambda}}$}-\ooalign{$ E_{\lambda}$\cr
	$\phantom{E}^{(\mathbf{n}){\scriptscriptstyle -}}$\cr
	$\phantom{E^{(\mathbf{n}){\scriptscriptstyle -}}_{\lambda}}$}|}
\ge \frac{(\frac{c}{C_{\mathsf{H}}})^{1/h}}{C(\lambda)e^{-C_{\mathsf{E}}\kappa_\lambda(\mathbf{n})}\lVert\mathbf{m}_\mathbf{n}-\mathbf{n}\rVert^{{(t-1)}/{h}}}
\ge \frac{(\frac{c}{C_{\mathsf{H}}})^{1/h}}{C(\lambda)e^{-C_{\mathsf{E}}\kappa_\lambda(\mathbf{n})}(2\kappa_\lambda(\mathbf{n}))^{{(t-1)}/{h}}}.
\end{align*}
As a result, $\textstyle \lim_{\lambda\to 0}\tau(\Sigma_{\lambda})=\lim_{\lambda\to0}\inf_{\flsubstack{\mathbf{n}\ne\mathbf{0}\\\hphantom{\mathbf{n}}\mathllap{\ooalign{$\scriptstyle G_{\scriptscriptstyle\lambda}$\cr
	$\phantom{\scriptstyle G}^{\scriptscriptstyle(\mathbf{n})}$\cr
	$\phantom{\scriptstyle G^{\scriptscriptstyle(\mathbf{n})}_{\scriptscriptstyle\lambda}}$}}\ne\varnothing}}\tau(\Sigma_\lambda,\ooalign{$ G_{\lambda}$\cr
	$\phantom{G}^{(\mathbf{n})}$\cr
	$\phantom{G^{(\mathbf{n})}_{\lambda}}$})=+\infty$.
\end{proof}
The proof of Lemma~\ref{lem: Proof of lem: Proof of thm: Main Theorem(2)(3)(1)} is a combination of three components.
For simplicity, we restrict this discussion to the case $ b=1 $.
The first component is that the frequency $ \alpha $ satisfies a Diophantine condition.
Specifically, there exists $ c>0 $ and there exists $ t>1 $ such that
\[ |q\alpha-p|\ge \frac{c}{|q|^{t-1}} \]
for every $ \frac{p}{q}\in\mathbb{Q} $.
The second component is that the integrated density of states $ \mathfrak{N}_\lambda $ is Hölder continuous for small couplings.
Specifically, there exists $ C_{\mathsf{H}}>0 $ and there exists $ h>0 $ such that
\[
|\mathfrak{N}_\lambda(x)-\mathfrak{N}(y)|\le C_{\mathsf{H}}|x-y|^{h}
\]
for sufficiently small $ \lambda $ and for every $ x,y $ and the constants $ C_{\mathsf{H}} $ and $ h $ are independent of the sufficiently small $ \lambda $.
The existence of $ C_{\mathsf{H}} $ and $ h $ is established by A.~Cai, C.~Chavaudret, J.~You, and Q.~Zhou in 2019 \cite{2019HolderContinuity}; see Theorem~\ref{thm: Integrated Density of States(3)} for the unabridged statement.
The third component is that the lengths of spectral gaps exponentially decay for large gap-labels.
Specifically, there exists $ C(\lambda)>0 $ and there exists $ C_{\mathsf{E}}>0 $ such that
\[
\mathrm{length}(\ooalign{$ G_{\lambda}$\cr
	$\phantom{G}^{(n)}$\cr
	$\phantom{G^{(n)}_{\lambda}}$})\le C(\lambda)e^{-C_{\mathsf{E}}|n|}
\]
for sufficiently small $ \lambda $ and for every $ n\ne 0 $ and the constant $ C_{\mathsf{E}} $ is independent of the sufficiently small $ \lambda $ and the constant $ C(\lambda) $ satisfies $ \lim_{\lambda\to0}C(\lambda)=0 $.
Here $ \ooalign{$ G_{\lambda}$\cr
	$\phantom{G}^{(n)}$\cr
	$\phantom{G^{(n)}_{\lambda}}$} $ is a spectral gap satisfying $ \mathfrak{N}_{\lambda}\upharpoonright \overline{\ooalign{$ G_{\lambda}$\cr
	$\phantom{G}^{(n)}$\cr
	$\phantom{G^{(n)}_{\lambda}}$}\vphantom{G^{(n)}}}\equiv n\alpha-\lfloor n\alpha\rfloor $ for every $ n\ne 0 $.
The existence of $ C_{\mathsf{E}} $ and $ C(\lambda) $ is established by M.~Leguil, J.~You, Z.~Zhao, and Q.~Zhou in 2017 \cite{2017preprintGapEstimate}; see Theorem~\ref{thm: Spectral Gaps(7)} for the unabridged statement.
Then we combine the three components to deduce
\[
\frac{\mathrm{length}([E_\lambda,F_\lambda])}{\mathrm{length}(\ooalign{$ G_{\lambda}$\cr
	$\phantom{G}^{(n)}$\cr
	$\phantom{G^{(n)}_{\lambda}}$})}\ge \frac{|F_\lambda-E_\lambda|}{C(\lambda)e^{-C_{\mathsf{E}}|n|}}\ge \frac{(\frac{1}{C_{\mathsf{H}}})^{\frac{1}{h}}|\mathfrak{N}_\lambda(F_\lambda)-\mathfrak{N}_\lambda(E_\lambda)|{}^{\frac{1}{h}}}{C(\lambda)e^{-C_{\mathsf{E}}|n|}}
\]
for every $ n\ne 0,E_\lambda,F_\lambda\in\Sigma_\lambda $ with $ E_\lambda<F_\lambda $.
If $ \ooalign{$ G_{\lambda}$\cr
	$\phantom{G}^{(n_1)}$\cr
	$\phantom{G^{(n_1)}_{\lambda}}$} $ is left-adjacent to $ [E_\lambda,F_\lambda] $ and $ \ooalign{$ G_{\lambda}$\cr
	$\phantom{G}^{(n_2)}$\cr
	$\phantom{G^{(n_2)}_{\lambda}}$} $ is right-adjacent to $ [E_\lambda,F_\lambda] $, then $ \mathfrak{N}_\lambda(E_\lambda)=n_1\alpha-\lfloor n_1\alpha\rfloor $ and $\mathfrak{N}_\lambda(F_\lambda)=n_2\alpha-\lfloor n_2\alpha\rfloor$ and
\[\textstyle|\mathfrak{N}_\lambda(F_\lambda)-\mathfrak{N}_\lambda(E_\lambda)|\ge \frac{c}{|n_2-n_1|^{t-1}} \]
and
\[
\frac{\mathrm{length}([E_\lambda,F_\lambda])}{\mathrm{length}(\ooalign{$ G_{\lambda}$\cr
	$\phantom{G}^{(n)}$\cr
	$\phantom{G^{(n)}_{\lambda}}$})}\ge \frac{(\frac{c}{C_{\mathsf{H}}})^{{1}/{h}}}{C(\lambda)e^{-C_{\mathsf{E}}|n|}|n_2-n_1|^{{(t-1)}/{h}}}.
\]
Because the frequency is fixed and because the spectrum is the same for any phase, we define $ \Sigma_\lambda\coloneqq \Sigma_{\lambda,\alpha,\omega} $.
The spectral-thickness $ \tau(\Sigma_\lambda) $ is the infimum-length-ratio of planks over gaps; more details can be found within \textsc{subsection~\ref{subsec: Thickness and Gap Lemmas}}.
Therefore it can be similarly shown that
\[
\tau(\Sigma_\lambda)\ge \inf_{\flsubstack{n\ne0\\\hphantom{n}\mathllap{\ooalign{$\scriptstyle G_{\scriptscriptstyle\lambda}$\cr
	$\phantom{\scriptstyle G}^{\scriptscriptstyle(n)}$\cr
	$\phantom{\scriptstyle G^{\scriptscriptstyle(n)}_{\scriptscriptstyle\lambda}}$}}\ne\varnothing}}\frac{(\frac{c}{C_{\mathsf{H}}})^{1/h}}{C(\lambda)e^{-C_{\mathsf{E}}\kappa_\lambda(n)}(2\kappa_\lambda(n))^{{(t-1)}/{h}}},
\]
where $ \kappa_\lambda(n) $ is a positive real-valued function of $ n $.
Taking the limit as $ \lambda $ approaches zero concludes the proof of Lemma~\ref{lem: Proof of lem: Proof of thm: Main Theorem(2)(3)(1)}.
\begin{thm}[Lemma~\ref{lem: Proof of thm: Main Theorem(2)(3)}]\label{thm: Proof of lem: Proof of thm: Main Theorem(2)(3)(2)}
Let $ H_{\lambda,\alpha,\omega}=\Delta+V_{\lambda,\alpha,\omega} $ be the almost Mathieu operator.
Fix $ \alpha\in\mathbb{T}\setminus\mathbb{Q} $.
Assume $ \alpha $ satisfies a Diophantine condition:
\[\textstyle \alpha\in\bigcup_{\substack{c>0\\t>1}}\bigcap_{\frac{p}{q}\in\mathbb{Q}}\left\{x\in\mathbb{R}:|qx-p|\ge \frac{c}{|q|^{t-1}}\right\}\mathrlap{{}\eqqcolon \mathsf{DC}.} \]
Then $ \lim_{\lambda\to 0}\tau(\sigma(H_{\lambda,\alpha,\omega}))=+\infty $.
\end{thm}
\begin{proof}[\textcolor{black}{Proof.}]\renewcommand{\qedsymbol}{\textcolor{black}{$ \blacksquare $}}%
Observe $ \cos(z) $ is holomorphic on $ \mathbb{C} $.
By Theorem~\ref{thm: Integrated Density of States(3)} for $ b=1 $ and Theorem~\ref{thm: Spectral Gaps(7)} for $ b=1 $ and Lemma~\ref{lem: Proof of lem: Proof of thm: Main Theorem(2)(3)(1)} for $ b=1 $, $ \left(\forall \alpha\in \mathsf{DC}\right)[\textstyle \lim_{\lambda\to 0}\tau(\sigma(H_{\lambda,\alpha,\omega}))=+\infty] $.
\end{proof}
\subsection*{Acknowledgements}
I want to thank Anton Gorodetski for his guidance and direction; this paper would not exist in its current form without the support.
I want to thank David Damanik and Svetlana Jitormirskaya for their feedback which resulted in a more polished version of this paper.
Along with the previous people mentioned, I want to thank Siegfried Beckus, Lingrui Ge, Ilya Kachovskiy, Abel Klein, and Wencai Liu  for the mathematics-related discussions that in some form or another positively influenced me and positively influenced this paper.

\end{document}